\documentclass[11pt]{amsart}

\usepackage[utf8]{inputenc}
\usepackage[english]{babel}
\usepackage{csquotes}		

\usepackage{geometry}
\usepackage{enumerate}
\usepackage{bbm}
\usepackage{xcolor}
\colorlet{green}{orange}
\newcommand{\RR}{{\mathbb R}}
\newcommand{\NN}{{\mathbb N}}

\renewcommand{\Mc}{\mathcal{M}}
\newcommand{\Hc}{\mathcal{H}}

\newcommand{\Mb}{\mathbf{M}}

\newcommand{\weakly}{\rightharpoonup}

\newcommand{\compact}{\subset\subset}
\newcommand{\rest}{\llcorner}
\DeclareMathOperator{\supp}{supp}
\DeclareMathOperator{\dive}{div}

\renewcommand{\ge }{\geq}
\renewcommand{\le }{\leq}

\newcommand{\de}{\partial}
\newcommand{\compl}{c}			

\newcommand{\eps}{\varepsilon}		
\let\oldphi\phi						
\newcommand{\phin}{\oldphi}
\renewcommand{\phi}{\varphi}

\renewcommand{\hat }{\widehat}
\renewcommand{\tilde }{\widetilde}

\newcommand{\hk}{\Psi}			
\DeclareMathOperator{\Clos}{Clos}		
\newcommand{\Gr}{\mathrm{Gr}}	

\DeclareMathOperator{\osc}{osc} 	

\newcommand{\tSigma}{\tilde\Sigma}  

\newcommand{\hhk}{\hat\hk}	
\newcommand{\Var}[1]{V_{{#1}}}	

\newcommand{\dif}{\mathop{}\!\mathrm{d}}	
\newcommand{\holder}{H\"{o}lder}
\newcommand{\trace}{\mathrm{trace}}

\newcommand{\Id}{\mathbbm{1}}
\newcommand{\transp}{\mathsf{T}}

\title{A short proof of Allard's and Brakke's regularity theorems}

\author[G.~De Philippis]{Guido De Philippis}
\address{\textit{G.~De Philippis:} Courant Institute of Mathematical Sciences, New York University, 251 Mercer St., New York, NY 10012, USA.}
\email{guido@cims.nyu.edu}

\author[C.~Gasparetto]{Carlo Gasparetto}
\address{\textit{C.~Gasparetto:} Dipartimento di Matematica, Università di Pisa, Largo Bruno Pontecorvo 5, 56127 Pisa, Italy}
\email{carlo.gasparetto@dm.unipi.it}

\author[F.~Schulze]{Felix Schulze}
\address{\textit{F.~Schulze:} Department of Mathematics, Zeeman Building, University of Warwick, Gibbet Hill Road, Coventry CV4 7AL,
UK}
\email{felix.schulze@warwick.ac.uk}

\thanks{{\bf Acknowledgments.} 
	G.D.P has been partially supported by the NSF grant DMS 2055686 and by the Simons Foundation. C.G. is supported by the European
	Research Council (ERC), under the European Union’s Horizon 2020 research and innovation program, through the project ERC
	VAREG - Variational approach to the regularity of the free boundaries (grant agreement No. 853404) and partially supported by INDAM-GNAMPA. Part of this work was carried out while C.G. was a PhD student at SISSA-Trieste}

\subjclass[2010]{}


\usepackage{mathtools,
	amsthm,
	amssymb		
}
\numberwithin{equation}{section}						
\renewenvironment{proof}[1][\proofname]{{\par\medskip\noindent\bfseries #1. }}{\qed\par}		

\usepackage[hypertexnames=false]{hyperref}		

\usepackage{cleveref}		
\crefname{subsection}{Subsection}{Subsections}		

\theoremstyle{plain}
\newtheorem{theorem}{Theorem}[section]
\newtheorem{proposition}[theorem]{Proposition}
\newtheorem{lemma}[theorem]{Lemma}

\newtheorem{definition}[theorem]{Definition}
\newtheorem{assumption}[theorem]{Assumption}
\theoremstyle{definition}
\newtheorem{remark}[theorem]{Remark}

\begin{document}
\maketitle

\begin{abstract}
{
We give new short proofs of Allard’s regularity theorem for varifolds with bounded first variation and Brakke’s regularity theorem for integral Brakke flows with bounded forcing. They are based on a decay of flatness, following from weighted versions of the respective monotonicity formulas, together with a characterization of non-homogeneous blow-ups using the viscosity approach introduced by Savin.
}
\end{abstract}

\section{Introduction}

Allard's and Brakke's {\(\varepsilon\)-regularity} theorems are key tools in the study of, respectively, minimal surfaces and mean curvature flows, and they can be roughly stated as follows: 

\medskip

\emph{If a \(m\)-dimensional minimal surface (resp.~the space-time {track} of a  mean curvature flow)  is sufficiently flat in the ball of radius \(1\) (resp.~parabolic cylinder of radius \(1\))  and its area  is roughly the one of the unit \(m\)-dimensional disk (resp.~the weighted Gaussian density of a unit disk) then in a smaller ball (resp.~parabolic cylinder) it can be written as the graph of a smooth function which enjoys suitable a-priori estimates.}

\medskip

See  Theorems \ref{thm:allard_main} and \ref{thm:brakke_interior}  below for the rigorous statement. Note {these results are also relevant} in the smooth category, since  the scale and the regularity of the graphical parametrization of the surface only depend on the a priori assumption of (weak) closeness to the \(m\)-dimensional unit disk.

The original proofs by Allard and Brakke are modeled on the pioneering ideas introduced by De Giorgi in the regularity theory for co-dimension \(1\) area minimizing surfaces, \cite{degiorgi61} and on their implementation done by Almgren in \cite{Almgren1968}. In particular the proof is, roughly speaking, divided into the following steps:

\begin{itemize}
\item[(a)] Under the desired assumption it is possible to show that most of the surface can be covered by the graph of a Lipschitz function, whose \(W^{1,2}\) norm can be estimated by the difference in area between the surface and the plane.
\item[(b)] Since the minimal surface equation (respectively the mean curvature flow)  linearizes on the Laplacian equation (resp.~the heat equation), this function is close to a harmonic (resp.~caloric) function which enjoys strong a priori estimates.
\item[(c)] These estimates can be pulled back to the minimal surface (resp.~space-time track of the mean curvature flow) to show that the  initial assumptions are satisfied also in the ball of radius {$1/2$}.  A suitable iteration provides then the conclusion. 
\end{itemize}
In this approach the two main difficulties  lie in the approximation procedure in step (a) and in proving that the closeness in step  (b) is in a strong enough topology to be able to pull-back the estimates from the linearized equation. Among various references, we refer the reader to \cite{DeLellis_allard} for a very clear account of Allard`s theorem and to \cite{kasai_tonegawa} for a simplified proof of Brakke theorem. Inspired by the work of Caffarelli and Cordoba \cite{caffarelli-cordoba}, in \cite{savin07} Savin provided a viscosity type approach to the above proofs, in which step (a) is completely avoided and step (b) is replaced by a partial Harnack inequality, obtained via Aleksandrov-Bakelmann-Pucci (ABP) type estimates, see \cite{savin2017viscosity} and \cite{wangsmall} for the extension to the minimal surface system and to parabolic equations, respectively. This approach has been recently exploited by the second named author to prove a boundary version of Brakke's regularity theorem, \cite{gaspa_brakke_bdry}. We also mention that  Savin's partial Harnack inequality has been a crucial ingredient in the proof of the De Giorgi conjecture on solutions of the Allen Cahn equation, \cite{Savin2009}.

In this note  we show how combining both the \enquote{variational} and the \enquote{viscosity} approach, it is possible to obtain a very short and {self-sufficient} proof of {both} Allard's and Brakke's theorems. The key observation is that while viscosity techniques are very robust in allowing to pass to the limit in the equation under \(L^\infty\) convergence (a key step in Savin's approach) the ABP estimate can be replaced by a simple variational argument based on the fact that coordinates are harmonic (resp.~caloric) when restricted to the minimal surface (resp.~mean curvature flow) and that for harmonic functions the mean value inequality can be easily obtained by testing the weak formulation of the equation with a suitable truncation of the fundamental solution, see for instance \cite{caffarelli1998obstacle}.

This short note is organized as follows: in \Cref{sec:allard}  we prove {Allard's} theorem and in \Cref{sec:brakke} we prove Brakke's theorem. In order to make the note {self-contained}, we conclude with an appendix where we {record} the proof of the maximum principle for varifolds and Brakke flows. Although the proofs are done in the natural  context of varifolds and Brakke flows, we invite the reader to take in mind the simple case of a smooth surface with zero mean curvature and {a smooth mean curvature flow}.

For the purpose of open access, the authors have applied a Creative Commons Attribution (CC- BY) license to any Author Accepted Manuscript version arising from this submission.

\section{Allard's regularity theorem}\label{sec:allard}
In the following, we denote by $\Mc_m(U)$ the space of {$m$-dimensional rectifiable Radon measures on $U$}, namely those Radon measures $M$ on $U$ for which there is a $m$-dimensional rectifiable set $E\subset U$ with $\Hc^m(E)<\infty$ and $M\ll\Hc^m\rest E$. If $M\in\Mc_m(U)$, then $M = \Theta^m(M,\cdot)\Hc^m$, where 
\begin{equation*}
	\Theta^m(M,x) = \lim_{r\searrow0}\frac{M(B_r(x))}{\omega_mr^m}
\end{equation*}
wherever the limit exists.
{For $M$-almost every $x$, the approximate tangent plane to $M$ at $x$ is well defined and we denote it by $T_xM\in\Gr(m,d)$.
For $S\in\Gr(m,d)$ and $F\in C^1(U;\RR^d)$, we introduce the notation
\begin{equation*}
	\dive_S F(x) = \sum_{i=1}^m\nabla F(x)\eta_i\cdot\eta_i,
\end{equation*}
where $\{\eta_i\}_{i=1}^m$ is any orthonormal basis of $S$.
Next, for every $p\in(1,+\infty]$, we let $\Mc_m^p(U)$ be the set of those measures $M\in\Mc_m(U)$ such that, for every $W\compact U$ there is $h_W\in \RR$ such that
\begin{equation}\label{eq:bounded_first_var}
	\int\dive_{T_xM}F(x)\dif M(x) \le h_W\left(\int|F|^{p'}\dif M\right)^{1/p'}
\end{equation}
for every $F\in C^1_c(W;\RR^d)$, where $p'$ is the {conjugate} exponent of $p$.
}To each $M\in\Mc_m^p(U)$ we associate a vector field $H_M\in L^p_{loc}(M;\RR^d)$ such that
\begin{equation}\label{eq:first_var}
	\int \dive_{T_xM} F\dif M = -\int H_M\cdot F\dif M
\end{equation}
for every $F\in C^1_c(U;\RR^d)$. We call $H_M$ the \textit{generalized mean curvature vector of $M$}. Whenever the indication of $M$ is unnecessary, we write $H$ in place of $H_M$. Before stating the main result of this section, we introduce the following notation: for $S\in\Gr(m,d)$, {$M\in\Mc^\infty_m(B_R)$ {and $x_0,r$ such that $B_r(x_0)\subset B_R$}, we let
\begin{equation}\label{eq:def_osc}
	\osc_S(M, B_r(x_0)) = \frac{1}{2}\sup\big\{|S^\perp(x-y)|\colon x,y\in \supp M\cap B_r(x_0)\big\}
\end{equation}
which corresponds to the radius of the smallest cylinder of the form $\{x\in\RR^d\colon|S^\perp(x-y)|\le h\}$ for some $y\in\RR^d$ that contains $\supp M\cap B_r(x_0)$.

The goal of the present section is proving the following version of Allard's theorem:
\begin{theorem}[Allard's regularity theorem]\label{thm:allard_main}
	For every $\alpha\in(0,1)$, there are $\delta_0>0$ and $C>0$ with the following property. Let $M\in\Mc_m^\infty(B_1)$ and assume that $0\in\supp M$, $\Theta^m(M,x)\ge1$ for $M$-almost every $x$,
	\begin{gather*}
		M(B_1)\le(1+\delta_0)\omega_m\qquad\mbox{and}\qquad		\|H_M\|_{L^\infty(M)}\le\delta_0.
	\end{gather*}

	Then $\supp M\cap B_{1/2}$ is the graph of some function $u\in C^{1,\alpha}(B_1^m;\RR^{d-m})$ with
	$$||u||_{C^{1,\alpha}}\le C\bigg(\inf_{S\in\Gr(m,d)}\osc_S(M,B_1)+||H||_{L^\infty(M)}\bigg).$$
\end{theorem}
{If $H$ is more regular, then Schauder estimates entail higher regularity for $\supp M$, as well.}
We shall prove \Cref{thm:allard_main} in \Cref{subsec:allard_IOF}. The next subsection is dedicated to proving a decay property of the oscillations of $\supp M$, which allows us to prove an improvement of flatness (see \Cref{thm:IOF_allard} below).

\subsection{Decay of oscillations}\label{subsec:allard_oscillation}
In the present subsection, we assume the following:
\begin{assumption}\label{ass:subsec_allard_harnack}
$M\in\Mc_m^\infty(B_R)$ is such that:
\begin{enumerate}
	\item \label{ass_allard_dens_below} $\Theta^m(M,x)\ge1$ for $M$-almost every $x \in B_R$;
	\item \label{ass:allard_dens_above}for every $r\in[0,R]$, $M(B_r)\le\frac{3}{2}\omega_mr^m$.
	\item \label{} $\Lambda := \|H_M\|_{L^\infty(M)}<\infty$.
\end{enumerate}
\end{assumption}

\begin{proposition}[Weighted monotonicity formula]\label{prop:weight_mono_allard}
{Let $f:B_R\to\RR$ be a non-negative, convex function such that $\|\nabla f\|_{\infty}\le 1$}. Provided $0\in \supp M$, then for every $0<r\le R$
\begin{equation*}
	\frac{1}{\omega_mr^m}\int_{B_r}f\dif M
	\ge f(0)- C_0 \Lambda\left(\|f\|_{L^\infty(M)}+r\right)r
\end{equation*} 
for some $C_0$ universal.
\end{proposition}
\begin{remark}
	Although this fact will not be used in the following, we point out that the above result holds true provided $f$ is $m$-convex, meaning that the sum of the $m$ smallest eigenvalues of $D^2f$ is non-negative.
\end{remark}
\begin{proof}
	We begin with some preliminary computations. Assume\footnote{In the case $m=2$, a logarithmic term replaces $|x|^{2-m}$.} {$m>2$} and let
	\begin{equation*}
		h(x)=\frac{1}{\omega_m m(m-2)}\begin{cases}
			\frac{m}{2}-\frac{m-2}{2}|x|^2\qquad&\mbox{if }|x|\le 1\\
			|x|^{2-m}&\mbox{if }|x|\ge1\, .
		\end{cases}
	\end{equation*}
{Note that $h \in C^{1,1}(\mathbb{R}^d)\cap C^2(\mathbb{R}^d\setminus \partial B_1)$ and
\begin{equation}\label{eq:nabla-h}
 \nabla h(x) = -\frac{1}{m\omega_m} 
 \begin{cases}
 x \qquad &\text{if}\ |x|< 1\\
 \frac{x}{|x|^m} &\text{if}\ |x|\ge1
\end{cases}
\end{equation}
as well as 
\begin{equation}\label{eq:nabla2-h}
 D^2 h(x) = -\frac{1}{\omega_m} 
 \begin{cases}
 \frac{1}{m} \Id \qquad &\text{if}\ |x|< 1\\
 |x|^{-m}\Big(\frac{1}{m}\Id -\frac{x\otimes x}{|x|^2}\Big) &\text{if}\ |x|>1\, .
\end{cases}
\end{equation}}
	For any $0<r\le s< R$, we then let 
	\begin{gather*}
		g_{r,s}(x)=r^{2-m}h(x/r)-s^{2-m}h(x/s).		
	\end{gather*}
	 Note that $g_{r,s} \geq 0$ and $g_{r,s}\equiv0$ outside $B_s$. Straightforward computations give, using \eqref{eq:nabla-h} and \eqref{eq:nabla2-h}, for any $S\in\Gr(m,d)$:
	\begin{equation*}
		\dive_S\nabla g_{r,s}(x) = -\frac{\chi_{B_r}}{\omega_mr^m}+\frac{\chi_{B_s}}{\omega_ms^m} + \frac{\chi_{B_s\setminus B_r}}{\omega_m|x|^m}\left(\frac{|Sx|^2}{|x|^2}-1\right)\le -\frac{\chi_{B_r}}{\omega_mr^m}+\frac{\chi_{B_s}}{\omega_ms^m}.
	\end{equation*}
	Moreover,  since $f$ is non-negative and {convex}
	\begin{equation}\label{computation_dive_fg}
		\begin{split}
			\dive_S(f\nabla g_{r,s} - g_{r,s}\nabla f)
			&= f\dive_S(\nabla g_{r,s}) - g_{r,s}\dive_S(\nabla f)\\
			&\le f\dive_S(\nabla g_{r,s})\\
			&\le -\frac{f\chi_{B_r}}{\omega_mr^m}+\frac{f\chi_{B_s}}{\omega_ms^m}.
		\end{split}
	\end{equation}
	Let now
	\begin{equation*}
		I(r)=\frac{1}{\omega_m r^m}\int_{B_r}f\dif M.
	\end{equation*}
	By \eqref{eq:first_var} and \eqref{computation_dive_fg}, for any $0<r<s\le R$, it holds
	\begin{align*}
		I(s)-I(r)
		\ge \int\dive_{T_xM}(f\nabla g_{r,s}-g_{r,s}\nabla f)\dif M
		=-\int H\cdot(f\nabla g_{r,s}-g_{r,s}\nabla f)\dif M.
	\end{align*}
	Therefore
	\begin{equation}\label{eq:pre_derivative}
		\frac{I(s)-I(r)}{s-r}\ge -\int f H\cdot\frac{\nabla g_{r,s}}{s-r}\dif M + \int \frac{g_{r,s}}{s-r} H\cdot\nabla f\dif M=:-A_1+A_2.
	\end{equation}
	
 	By \holder's inequality, we may estimate
	\begin{align*}
		|A_1|\le \Lambda\sup\left|\frac{\nabla g_{r,s}}{s-r}\right|\int_{B_s}f\dif M\le C\Lambda \sup\left|\frac{\nabla g_{r,s}}{s-r}\right| s^m I(s)
	\end{align*}
	and
	\begin{align*}
		|A_2|&\le\Lambda \sup \left|\frac{g_{r,s}}{s-r}\right|\|\nabla f\|_{\infty} M(B_s) \le C\Lambda \sup \left|\frac{g_{r,s}}{s-r}\right| s^{m}
	\end{align*}
	where in the second inequality we have used the fact that $\|\nabla f\|_\infty\le1$ and \Cref{ass:allard_dens_above} in \Cref{ass:subsec_allard_harnack} above.
	Direct computations give, using \eqref{eq:nabla-h} and \eqref{eq:nabla2-h} together with the convexity of $t\mapsto t^{-m}$, 
	\begin{gather*}
		{\limsup_{s\searrow r}}\sup_{x\in B_R}\left|\frac{\nabla g_{r,s}}{s-r}\right|\le Cr^{-m},\qquad
		{\limsup_{s\searrow r}}\sup_{x\in B_R}\left|\frac{g_{r,s}}{s-r}\right|\le Cr^{1-m}
	\end{gather*}
	for some $C$ depending only on $m$. Thus, letting $s\searrow r$ in \eqref{eq:pre_derivative}, {together with standard approximation arguments (see for example \cite[\S 17]{simonGMT})}, gives
	\begin{equation*}
		I'(r)\ge-C\Lambda (I(r)+r)
	\end{equation*}
	in the sense of distributions.
	We now multiply both sides of the latter inequality by the integrating factor $F(r)=e^{C\Lambda r}$
	and integrate from some $h>0$ to $r$ to obtain
	\begin{align*}
		I(r)
		\ge e^{C\Lambda(h-r)}I(h)-C\Lambda\int_{h}^re^{C\Lambda(t-r)}t\dif t
		\ge I(h) - C\Lambda r\|f\|_{L^\infty(M)}-C\Lambda r^2,
	\end{align*}
	where in the second inequality we have used the facts that $e^t\ge1+t$ and that, by \Cref{ass:allard_dens_above} in \Cref{ass:subsec_allard_harnack}, $I(h)\le C\|f\|_{L^\infty(M)}$.
	Finally, if $\Theta^m(M,0)\ge1$, letting $h\searrow0$ in the above inequality gives the desired result. Otherwise, if $0\in\supp M$ then there is a sequence $x_j\to0$ such that $\Theta^m(M,x_j)\ge1$. The result follows by continuity of $f$.
\end{proof}

\begin{remark}
	The role of the {factor  $3/2$ in Assumption \ref{ass:subsec_allard_harnack} (2)}, as will be clear from the proof of \Cref{harnack_allard} below, is to rule out the possibility that $M$ consists of two separated sheets, which would clearly violate the conclusion of a Harnack inequality.
	On the other hand, for what concerns \Cref{prop:weight_mono_allard}, by virtue of the classical monotonicity formula (see, for instance, \cite[\S 17]{simonGMT}), (2) in \Cref{ass:subsec_allard_harnack} may be replaced by the weaker assumptions
	\begin{equation*}
		M(B_R)\le E_0R^m\mbox{ and }\Lambda\le\frac{E_0}{R},
	\end{equation*}
	with the caveat that $C_0$ in the conclusion of \Cref{harnack_allard} depends on $E_0$.
\end{remark}

The above result allows us to prove a \textit{partial Harnack inequality}. 
We refer the reader to \eqref{eq:def_osc} for the definition of $\osc_S(M,B_r(x_0))$. Whenever {$x_0=0$} and the indication of $M$ and of the $m$-plane $S$ is unnecessary, as it is in the next two results, {we write} $\osc(B_r)$ in place of $\osc_S(M,B_r(x_0))$.

\begin{proposition}[Harnack inequality]\label{harnack_allard}
	Let $M$ satisfy \Cref{ass:subsec_allard_harnack} {and let $0 \in \supp M$}.
	There exists a universal constant {$\eta \in (0,1)$} such that, if
	\begin{equation*}
		\osc(B_R)\le\eta R\qquad\mbox{ and }\qquad\Lambda\le \frac{\osc(B_R)}{R^2},
	\end{equation*}
	then
	\begin{equation*}
		\osc(B_{\eta R})\le (1-\eta)\osc(B_R).
	\end{equation*}
\end{proposition}
\begin{proof}
	Let $\eps=\frac{1}{R}\osc(B_R)$ and $\Sigma=\supp M$.  Since $0 \in \supp M$, we can choose $y_0 \in \overline B_{\eps R}$ such that $\Sigma\cap B_R\subset\{y\colon |S^\perp (y-y_0)|\le \eps R\}$. Assume, by contradiction, that there are points $y_1, y_2 \in \Sigma\cap B_{\eta R}$ such that $|S^\perp (y_1-y_2)|> 2 (1-\eta)\eps R$. Denote $\omega:= \frac{S^\perp (y_1-y_2)}{|S^\perp (y_1-y_2)|}$ and consider 
	\begin{equation*}
		f_1(x) = \left((x-y_0)\cdot \omega -\frac{\eps R}{2}\right)^+\quad \text{and} \quad {f_2(x) = \left(-(x-y_0) \cdot \omega-\frac{\eps R}{2} \right)^+}
	\end{equation*}
	and let
	\begin{equation*}
		I_i(r) = \frac{1}{\omega_mr^m}\int_{B_r(y_i)}f_i\dif  M.
	\end{equation*}
	Notice that, by assumption, {$\|f_i\|_{L^\infty( M\cap B_R)}\le\frac{\eps}{2}R$} and $\Lambda\le \eps R^{-1}$. Moreover,  $f_i(y_i)\ge\left(\frac{1}{2}-2\eta\right)\eps R$. Therefore, by \Cref{prop:weight_mono_allard}, for every $\theta \in [\eta,1/2]$ (to be determined later) it holds
	\begin{equation}
		\begin{split}
		I_i(\theta R)
		&\ge f_i(y_i) - C_0\Lambda \theta R\big(\|f_i\|_{L^\infty(M)}+\theta R\big)\\
		&\ge \Big(\frac{1}{2}-2\eta\Big)\eps R - C_0\eps\theta \Big(\frac{\eps}{2}+\theta\Big)R	.	
		\end{split}\label{eq:only_one}
	\end{equation}
Since, by assumption, $\eps\le \eta \le \theta \le \tfrac{1}{2}$ we have
{$ C_0\eps\theta \Big(\frac{\eps}{2}+\theta\Big) \le 2 C_0 \theta^2 \eps \leq C_0\theta \eps$.}
We thus obtain from \eqref{eq:only_one} 
\begin{equation}\label{eq:only_two}
 I_i(\theta R) \geq \Big(\frac{1}{2}-2\eta - C_0\theta\Big)\eps R \geq \Big(\frac{1}{2}-(C_0+2)\theta\Big)\eps R.
\end{equation}
	 Note that $\supp f_1\cap\supp f_2=\emptyset$. Using this fact and the assumptions $M(B_r)\le\frac{3}{2}\omega_mr^m$ and $|y_1|,|y_2|\le2\eps R\le2\eta R$, we have
	 \begin{align*}
		I_1(\theta R)+I_2(\theta R)& \le \frac{1}{\omega_m(\theta R)^m}\int_{B_{\theta R}(y_1)\cup B_{\theta R}(y_2)} f_1 +f_2\, dM\\
		&\le\frac{\eps R}{2}\frac{M(B_{\theta R+2\eta R})}{\omega_m(\theta R)^m}\\
		&\le \frac{3}{4}\Big(1+\frac{2\eta}{\theta}\Big)^m\eps R \leq \frac{3}{4}\Big(1+C_1\frac{\eta}{\theta}\Big)\eps R
	\end{align*}
	for some $C_1$ depending only on $m$. We now specify our choices of $\theta$ and $\eta$. We first choose $\theta$ much smaller than $\min\{C_1^{-1}, (C_0+2)^{-1}\}$ such that
	$$ 2\Big(\frac{1}{2}-(C_0+2)\theta\Big) > \frac{3}{4}\Big(1+C_1\theta\Big)\, .$$
	Fixing $\eta \leq \theta^2$ we obtain a contradiction from \eqref{eq:only_two} and summing for $i\in\{1,2\}$.
	\end{proof}
\bigskip
As a corollary, we have the following
\begin{proposition}\label{prop:allard_decay_final}
	There are positive universal constants $\beta$ and $C$ with the following property. Let $M$ satisfy \Cref{ass:subsec_allard_harnack} {and $0\in \supp M$}. Then, for every $r\in [\osc(B_R)+\Lambda R^2, R]$, it holds
	\begin{equation*}
		\osc(B_r)\le C (\osc(B_R)+\Lambda R^2) \left(\frac{r}{R}\right)^\beta.
	\end{equation*}
\end{proposition}
\begin{proof}
	By rescaling, assume $R=1$.	For some $K$ large to be determined later, let
	\begin{equation*}
		F(r) = \osc(B_r)+K \Lambda r^{2}.
	\end{equation*}
	We claim that, if $\eta$ is the constant determined in \Cref{harnack_allard}, then for any $r$ such that $\osc(B_r)\le\eta r$ it holds
	\begin{equation}\label{eq:decay_step1}
		F(\eta r)\le (1-\eta) F(r).
	\end{equation}
	Indeed, if $\Lambda\le \osc(B_r) r^{-2}$, then \Cref{harnack_allard} yields $\osc(B_{\eta r})\le (1-\eta)\osc(B_r)$, hence
	\begin{align*}
		F(\eta r) \le (1-\eta)\osc(B_r) + \eta^{2}K\Lambda r^{2}\le (1-\eta)F(r)
	\end{align*}
	since $\eta$ is much smaller than $1$. Otherwise, if $\osc(B_r)\le \Lambda r^2$, then
	\begin{align*}
		F(\eta r) &\le \osc(B_r) + K\Lambda \eta^{2}r^{2}\\
		&\le K\Lambda r^{2}\left(\frac{1}{K} + \eta^{2}\right)\\
		&\le (1-\eta)F(r)
	\end{align*}
	as desired, provided $K\ge\frac{1}{1-\eta-\eta^{2}}$.
	
	By induction, if  $\eta^{k}\ge F(1)$, then
	\begin{equation*}
		\osc(B_{\eta^k})\le F(\eta^k)\le (1-\eta)^k F(1).
	\end{equation*}
	Now, for {$F(1)\le r< 1$, let $k\in\NN \cup \{0\}$} be such that $\eta^{k+1}\le r<\eta^k$. Then we have
	\begin{equation*}
		{\osc(B_r)\le \osc(B_{\eta^k})\le (1-\eta)^kF(1) \leq \frac{1}{1-\eta}\eta^{\beta(k+1)}F(1)\le Cr^\beta F(1)}
	\end{equation*}
	provided $\beta$ is chosen small so that {$\eta^\beta\ge(1-\eta)$}
	and $C\ge\frac{1}{1-\eta}$.
\end{proof}

\subsection{Allard's theorem}\label{subsec:allard_IOF}
We now prove the following

\begin{theorem}[Improvement of flatness]\label{thm:IOF_allard}
	For every $\alpha\in(0,1)$, there are {positive} constants $c,\eps_0,\eta$ and $C$ with the following property.
	Let $M\in\Mc_m^\infty(B_R(x_0))$ be such that
	\begin{itemize}
		\item $x_0\in\supp M$;
		\item $\Theta^m(M,x)\ge1$ for $M$-almost every $x$;
		\item for every $x\in B_{R/2}(x_0)$ and $0<r\le R/2$ it holds $M(B_r(x))\le \frac{3}{2}\omega_mr^m$;
	\end{itemize}
	and, for some $\eps\le\eps_0$ and $S\in\Gr(m,d)$,
	\begin{gather*}
			\osc_S(M,B_R(x_0))\le\eps R\qquad\mbox{ and }\qquad
			\|H_M\|_{L^\infty(M)}\le c \eps R^{-1}.
	\end{gather*}
	Then there is $T\in\Gr(m,d)$ with $|S-T|\le C\eps$ such that
	\begin{equation}\label{allard_iof_concl}
		\osc_T(M,B_{\eta R}(x_0))\le \eta^{1+\alpha}\eps R.
	\end{equation}
\end{theorem}
Before proceeding with the proof, we introduce the following notation. To every $M\in\Mc_m(U)$ we associate a rectifiable varifold $V_M$ such that $\int\phi(x,S)\dif V_M(x,S) = \int\phi(x,T_xM)\dif M(x)$ for every $\phi\in C^0_c(U\times\Gr(m,d))$. If $M\in\Mc^p_m(U)$ and $H_M=0$, we say that $V_M$ is \textit{stationary}.

\begin{proof}
	By rescaling and translating, we assume $R=1$ and $x_0=0$.
	We argue by contradiction and compactness. Assume there are sequences $\eps_j\searrow 0, c_j\searrow0$ and $\{M^j\}\subset\Mc^\infty_m(B_1)$ such that the assumptions of the theorem are satisfied with $\eps_0$ replaced by $\eps_j$ and $c$ replaced by $c_j$, for which however \eqref{allard_iof_concl} fails for any choice of $\eta$ and $T$.
	
	Before proceeding, we remark that, by compactness (see, for instance, \cite[Theorem 42.7]{simonGMT}), there is $M\in\Mc_m(B_1)$ such that  $0\in\supp M$, $\Theta^m(M,x)\ge1$ for $M$-a.e.~$x$, $M(B_1\setminus S)=0$, $\Var{M}$ is stationary and $\Var{M^j}\weakly\Var{M}$. Therefore $M=\Theta^m(M,x)\Hc^m\rest S$ and $S\ni x\mapsto \Theta^m(M,x)$ has vanishing distributional gradient, thus $M =\Theta_0\Hc^m\rest S$ for some $\Theta_0\ge1$.
	
	Let now $\Sigma^j=\supp M^j$;
	for every $x\in\RR^d$, define $F_j(x)=(Sx,\eps_j^{-1}S^\perp x)\in\RR^d$ and let
	\begin{equation*}
		\tSigma^j = F_j(\Sigma^j)\subset B_1^S\times \overline{B_2^{S^\perp}}.
	\end{equation*}
	Notice that $\tSigma^j$ are relatively closed subsets of $B_1^S\times \overline{B_2^{S^\perp}}$ and that they are non-empty, since by assumption $0\in\tSigma^j$ for every $j$.
	Therefore there is a relatively closed set $\tSigma$ in $B_1^S\times \overline{B_2^{S^\perp}}$ such that, up to subsequences, $\tSigma^j$ converges to $\tSigma$ in the Hausdorff distance.
	
	\newcommand{\ub}{\mathbf{u}}
	For every $x'\in B_1^S$, we let
	\begin{equation*}
		\ub(x')=\{y\in S^\perp\colon (x',y)\in \tSigma\}.
	\end{equation*}
	First, we show that $\ub(x')\neq\emptyset$ for every $x'\in B_1^S$. Indeed, otherwise, there would be $r>0$ such that for every $j$ large enough 
	\begin{equation*}
		\tSigma^j\cap (B^S_r(x')\times \overline{B_2^{S^\perp}}) =\emptyset,
	\end{equation*}
	hence, by compactness and lower semicontinuity of the mass,
	\begin{equation*}
		0=\liminf_jM^j(B^S_r(x')\times \overline{B_2^{S^\perp}})\ge M(B^S_r(x')\times \overline{B_2^{S^\perp}})\ge\omega_mr^m,
	\end{equation*}
	which is false.
	
	Secondly, we prove that $\ub(x')$ is a singleton for every $x'\in B_{1/2}^S$. Indeed, by \Cref{prop:allard_decay_final}, for every $j$, every $x\in\Sigma^j\cap B_{1/2}$ and every $r\in\left(C\eps_j,\frac{1}{2}\right)$, it holds
	\begin{equation*}
		\osc_S(M^j,B_r(x))\le C\eps_j r^\beta.
	\end{equation*}
	Therefore, by Hausdorff convergence, for every $x,y\in\tSigma\cap (B_{1/2}^S\times \overline{B_2^{S^\perp}})$ it holds
	\begin{equation*}
		|S^\perp(x-y)|\le C |S(x-y)|^\beta.
	\end{equation*}
	In particular, $\ub(x')$ is a singleton; for the rest of the proof, we denote by $u(x')$ the only element of $\ub(x')$. {Note further that this implies that $u \in C^{0,\beta}(B^S_{1/2})$.}
	By \Cref{lemma:uisharmonic} below, $u$ is harmonic in $B^S_{1/4}$. In particular, since $\sup_{B^S_{1/4}}|u|\le 2$, classical elliptic estimates yield
	\begin{equation*}
		{\sup_{B^S_{1/8}}\left(|\nabla u|+|D^2u|\right)}\le C
	\end{equation*}
	for some $C$ universal. Since $u(0)=0$, we may choose $\eta$ small  depending only on $C$ {and $\alpha$} so that
	\begin{equation*}
		{|u(x')-u(0)-\nabla u(0) x'|\le C|x'|^2\le \frac{ \eta^{1+\alpha}}{2}}
	\end{equation*}
	for every $x'\in B_{2\eta}^S$. However, this implies that, for every $j$ large enough and every $x\in\Sigma^j\cap B_\eta$, it holds
	\begin{equation*}
		|S^\perp x - \eps_j\nabla u(0)\,(Sx)|\le \eps_j\eta^{1+\alpha},
	\end{equation*}
	which contradicts the assumptions made at the beginning of the proof.
\end{proof}
\begin{lemma}\label{lemma:uisharmonic}
	$u$ defined in the proof of \Cref{thm:IOF_allard} is harmonic.
\end{lemma}
\begin{proof}
	We argue as in \cite[Lemma 2.4]{savin2017viscosity}.
	Let $h:B^S_{1/4}\to S^\perp$ be the harmonic function such that $(h-u)|_{\de B^S_{1/4}}=0$. If $u\neq h$, then there is $0<\delta<1/2$ small such that, for all $j$ large enough, the function
	\begin{equation*}
		G_j(x) = \frac{1}{2}\left|\frac{S^\perp x}{\eps_j}-h(Sx)\right|^2+\frac{\delta}{2}|Sx|^2
	\end{equation*}
	is such that $G_j|_{\Sigma^j}$ achieves its maximum at some point $x_j$ with $|Sx_j|\le\frac{1}{4}-\delta$. 
	We claim that $j$ can be chosen so large that, for every $T\in\Gr(m,d)$, it holds
	\begin{equation}\label{eq:conclharm}
		\dive_T\nabla G_j(x_j)> \|H_{M^j}\|_{L^\infty}|\nabla G_j(x_j)|,
	\end{equation}
	which would contradict \Cref{prop:max_principle}.
	In the rest of the proof, $C$ denotes constants (whose value may change from one expression to the other) which depend only on $d,m$ and $\delta$, but they are independent of $j$.
	We start by noticing that, by standard elliptic estimates, $$\max\{|\nabla h(Sx_j)|,|D^2h(Sx_j)|\}\le C.$$ Therefore $|\nabla G_j(x_j)| \le \frac{C}{\eps_j}$ for $j$ large enough and $\|H_{M^j}\|_\infty|\nabla G_j(x_j)|\le Cc_j\to0$ as $j\to\infty$. Thus, in order to prove \eqref{eq:conclharm}, it is sufficient to prove that {for $j$ sufficiently large
	\begin{equation}\label{eq:conclharm1}
		\inf_{T\in\Gr(m,d)}\dive_T\nabla G_j(x_j)\ge \delta.
	\end{equation}
	}
	Let $f_j(x) = \frac{1}{\eps_j}S^\perp x-h(Sx)$ and let $\Id_S$ denote the orthogonal projection onto $S$. Then 
	\begin{align*}
		D^2G_j(x_j) &= D^2f_j(x_j)\cdot f_j(x_j) + \nabla f_j(x_j)(\nabla f_j(x_j))^\transp + \delta\Id_S \\
		&\ge D^2f_j(x_j)\cdot f_j(x_j) +\delta\Id_S \\
		&=:A_j.
	\end{align*}
	In particular, since $\Delta h=0$,
	\begin{equation*}
		\dive_S\nabla G_j(x_j)\ge \trace_SA_j={\Delta h(Sx_j)}\cdot\left(\frac{S^\perp x_j}{\eps_j}-h(Sx_j)\right)+m\delta= m\delta,
	\end{equation*}
	{where $\trace_SA_j = \sum_{i=1}^mA_j\xi_i\cdot\xi_i$ for any orthonormal basis $\{\xi_i\}_{i=1}^m$ of $S$.}
	Since $|A_j|\le C$, by continuity it holds $\dive_T\nabla G_j(x_j)\ge \delta$ for every $T\in\Gr(m,d)$ such that {$|T-S|\le \gamma$ for $\gamma > 0$ sufficiently small. This proves \eqref{eq:conclharm1} in the case $|T-S|\le \gamma$.}
	
	On the other hand, if {$|T-S|\ge \gamma$}, then there is a unit vector $\eta\in T$ such that {$|S^\perp \eta|\ge c_0\gamma$ for some $c_0>0$ depending only on the dimension}. Then
	\begin{equation*}
		\dive_T\nabla G_j(x_j)\ge {\text{trace}_TA_j} + |(\nabla f_j(x_j))^\transp\eta|^2.
	\end{equation*}
	We have $|\text{trace}_TA_j|\le|A_j|\le C$ and
	\begin{equation*}
		\left|(\nabla f_j(x_j))^\transp \eta \right|= \left|-\nabla h(Sx_j) S\eta + \frac{1}{\eps_j} S^\perp \eta\right|\ge \frac{1}{\eps_j}|S^\perp\eta|-|\nabla h(Sx_j)|\ge {\frac{c_0 \gamma}{\eps_j} - C},
	\end{equation*} 
	thus {$\dive_T\nabla G_j(x_j)\ge  (\frac{c_0 \gamma}{\eps_j} - C)^2 - C \ge \delta$} for $j$ large enough. This proves \eqref{eq:conclharm1} in the case {$|T-S|\ge \gamma$}.
%
\end{proof}
\bigskip

Notice that the conclusion of \Cref{thm:IOF_allard} may be iterated at all scales. In particular, if we let
\begin{equation*}
	E(M,B_r(x)) = \inf_{S\in\Gr(m,d)}\frac{\osc_S(M,B_r(x))}{r} + C\|H\|_{L^\infty(M\rest B_r)}r
\end{equation*}
for some $C$ large, then \Cref{thm:IOF_allard} yields the existence of some universal constants $\eta$ and $\eps_0$ such that
\begin{equation*}
	E(M,B_{\eta R})\le \eta^{\alpha} E(M,B_R)
\end{equation*}
provided $E(M,B_R)\le\eps_0$.
A straight-forward induction argument yields, for every $r>0$,
\begin{equation}\label{eq:excess_decay}
	\inf_{S\in\Gr(m,d)}\frac{\osc_S(M,B_r)}{r}\le E(M,B_r)\le C\left(\frac{r}{R}\right)^{\alpha}E(M,B_R)
\end{equation}
for some $C$ universal, provided $E(M,B_R)\le\eps_0$.

\bigskip
We finally prove \Cref{thm:allard_main}.
\begin{proof}[Proof of \Cref{thm:allard_main}]
	\textbf{Step 1. }We claim that, if $\delta_0$ is small enough, then there exists $S\in\Gr(m,d)$ such that the assumptions of \Cref{thm:IOF_allard} are in place for $R=\frac{1}{4}$ and any $x_0\in \supp M\cap B_{3/4}$. We argue by compactness and contradiction: consider sequences $\delta_j\searrow0$ and $\{M^j\}\subset \Mc_m^\infty(B_1)$ that satisfy the assumptions of \Cref{thm:allard_main} with $\delta_0$ replaced by $\delta_j$.
	
	We first remark that, up to subsequences, there exists $M\in\Mc_m(B_1)$ such that $\Theta^m(M,x)\ge1$ $M$-almost everywhere, $V_M$ is stationary and $M(B_1)\le\omega_m$. By \cite[Theorem 5.3]{allard1972}, we have that $M=\Hc^m\rest S$ for some $S\in\Gr(m,d)$.
	
	Given $\alpha\in(0,1)$, let now $\eps_0$ be the constant given in \Cref{thm:IOF_allard}. Then, for $j$ large enough,
	\begin{equation*}
		\supp M^j\subset\{y\colon |S^\perp y|\le \eps_0/2\}.
	\end{equation*}
	The only thing left to prove is that
	\begin{equation}\label{eq:small_1/2}
		M^j(B_{1/4}(x))\le\left(1+\frac{1}{4}\right)\frac{\omega_m}{4^m}
	\end{equation}
	for every $j$ large and every $x\in\supp M^j\cap B_{3/4}$. Given the above inequality, by monotonicity (\cite[\S 17]{simonGMT}), provided $\delta_0$ is smaller than some universal constant, we obtain $M^j(B_{r}(x))\le\frac{3}{2}\omega_mr^m$
	for every $0\le r\le \frac{1}{4}$, as required in \Cref{thm:IOF_allard}.
	
	We now prove \eqref{eq:small_1/2}. If the result is false, then there exist a subsequence $j_k$ and points $x_k\in\supp M^{j_k}\cap B_{3/4}$ such that
	\begin{equation}\label{eq:contr0}
		M^{j_k}(B_{1/4}(x_k))\ge\left(1+\frac{1}{4}\right)\frac{\omega_m}{4^m}.
	\end{equation}
	Up to extracting a further subsequence, $x_k\to x\in \overline{B_{3/4}}$ and, by monotonicity, $x\in\supp M$. Therefore, for any $\eps>0$,
	\begin{equation*}
		\omega_m\left(\frac{1}{4}+\eps\right)^m=M\bigg({B_{\frac{1}{4}+\eps}(x)}\bigg)\ge\limsup_kM^{j_k}\bigg({B_{\frac{1}{4}+\eps}(x)}\bigg)\ge\limsup_kM^{j_k}(B_{1/4}(x_k))
	\end{equation*}
	contradicting \eqref{eq:contr0} and thus proving the claim.
	
	\textbf{Step 2. }We now prove that $\supp M\cap B_{3/4}$ is the graph of some function $u\colon S(\supp M)\to \RR^{d-m}$. For the rest of the proof, we let
	\begin{equation*}
		\eps:=\osc_S(M,B_1)+||H_M||_{L^\infty(M)},
	\end{equation*}
	where $S\in\Gr(m,d)$ was determined in the previous step, and we set $\Sigma:=\supp M$. By iterating \Cref{thm:IOF_allard}, for every $x\in\Sigma\cap B_{3/4}$ we find $T_x\in\Gr(m,d)$ such that 
	\begin{equation}\label{eq:planes_close}
		|T_x-S|\le C\eps
	\end{equation}
	and, for every $r\le\frac{1}{4}$,
	\begin{equation*}
		\osc_{T_x}(M,B_r(x))\le C\eps r^{1+\alpha}.
	\end{equation*}
	It then follows that, for any two $x,y\in\Sigma\cap B_{3/4}$, it holds
	\begin{equation*}
		|T_x^\perp(x-y)|\le C\eps|S(x-y)|^{1+\alpha}.
	\end{equation*}
	together with \eqref{eq:planes_close}, the above inequality shows at once that there is $u:S(\Sigma\cap B_{3/4})\to \RR^{d-m}$ such that
	\begin{equation*}
		\Sigma\cap B_{3/4} = \{x\in\RR^{d}\colon S^\perp x = u(Sx)\}
	\end{equation*}
	and that, for every $x'\in S(\Sigma)$, there is a linear function $L_{x'}:S\to S^\perp$ such that, for every $y'\in S(\Sigma)$, it holds
	\begin{equation}\label{eq:almost_holder}
		|u(y')-u(x')-L_x(y'-x')|\le C\eps|x'-y'|^{1+\alpha}.
	\end{equation}

\textbf{Step 3. }We conclude the proof by showing that $S(\Sigma)\supset B_{1/2}^m$. Once that is proved, \eqref{eq:almost_holder} gives $||u||_{C^{1,\alpha}}\le C\eps$, as desired. We argue by contradiction: if $B_{1/2}^m\setminus S(\Sigma)\neq\emptyset$, since $S(\Sigma)$ is relatively closed in $B_{1/2}^m$ and $0\in\supp M$ by assumption, there must be a ball $B^m_r(x')\subset B_{1/2}^m\setminus S(\Sigma)$ and a point $y'\in\partial B^m_r(x')\cap S(\Sigma)$. If $\eps$ is smaller than some universal constant, then $y=(y',u(y'))\in\Sigma\cap B_{3/4}$. Consider a blow-up sequence
\begin{equation*}
	M^i = \frac{M-y}{r_i}
\end{equation*}
with $r_i\searrow 0$. Up to subsequences, $M_i$ converges weakly to some $M^\infty\in\Mc_m^\infty(B_1)$ with $V_{M^\infty}$ stationary. Since $\Sigma\cap (B_r^m(x')\times S^\perp)=\emptyset$, \eqref{eq:almost_holder} yields that $\supp M^\infty$ is included in a $m$-dimensional half-plane. This contradicts the Constancy Theorem (see, for instance, \cite[Theorem 41.1]{simonGMT}), concluding the proof.
\end{proof}

\section{Brakke's regularity theorem}\label{sec:brakke}
In this section, we show how the arguments in the previous section may be adapted to the case of mean curvature flows.
{Inspired by \cite{kasai_tonegawa}, we give the following definition.}

\begin{definition}[Brakke flow with transport term]\label{def:BF}
	We say that a family $\Mb=\{M_t\}_{t\in [0,\Omega]}$ of Radon measures on $U\subset\RR^d$  is a \textit{Brakke flow with transport term $v$ in $U\times[0,\Omega]$} if the following hold true:
	\begin{enumerate}
		\item \label{item:Mt_is_rect} for almost every $t\in [0,\Omega]$, $M_t\in \Mc^2_m(U)$ {and, for $M_t$-almost every $x$, $\Theta^m(M_t,x)$ is a positive integer.}
		\item For every $W\compact U$, 
		\begin{equation*}
			\int_0^\Omega{\int_W}\left(1+|v|^2+|H_{M_t}|^2\right)\dif M_t\dif t<\infty.
		\end{equation*}
		\item For every non-negative test function $\phin\in C^1_c(U\times [0,\Omega])$ and for every $[t_1,t_2]\subset [0,\Omega]$, it holds
		\begin{align}\label{eq:def_brakke}
			M_{t_2}(\phin(\cdot,t_2))-M_{t_1}(\phin(\cdot,t_1))\le\int_{t_1}^{t_2}\int\left(\de_t\phin + (-\phin H+\nabla\phin)\cdot(H+v^\perp)   \right)\dif M_t\dif t,
		\end{align}
		where $H$ is the generalized mean curvature vector of $M_t$, {as defined in \eqref{eq:first_var},} and $v^\perp(x) = (T_xM_t)^\perp v(x)$ for $M_t$-almost every $x$.
	\end{enumerate}
\end{definition}
For a Brakke flow with transport as above, we define the measure $M$ on $U\times[0,\Omega]$ by $\int\phi(x,t)\dif M(x,t) = \int\int\phi(x,t)\dif M_t(x)\dif t$
and the space-time track of $\Mb$:
\begin{equation*}
	\Sigma_{\Mb} = \Clos\Bigg(\bigcup_{t\in [0,\Omega]}\supp M_t\times\{t\}\Bigg).
\end{equation*}

\begin{remark}
	By virtue of \cite[Chapter 5]{brakke}, it holds
	\begin{equation}\label{eq:ortho}
		H_{M_t}(x)\perp T_xM_t
	\end{equation}
	at $M$-almost every $(x,t)$. This fact is used in \Cref{prop:mono_brakke} below.	 
\end{remark}

\subsection{Decay of oscillations}
In the following, we denote by $Q_R(x_0,t_0)$ the {backwards} parabolic cylinder $B_R(x_0)\times[t_0-R^2,t_0]$.
The following assumptions will be used in the present subsection.
\begin{assumption}\label{ass:brakke_oscil_interno}
	\hfill
	\begin{enumerate}[(1)]
		\item $\Mb$ is a Brakke flow with transport term $v$ in $Q_R$ and $(0,0)\in\Sigma_\Mb$;
		\item \label{assumption:MDR} there is $E_1<\infty$ such that, for every $(x,t)\in Q_R$ and $B_r(x)\subset B_R$, it holds $M_t(B_r(x))\le E_1r^m$;
		\item \label{ass:norm_v} $\Lambda:=\|v\|_{L^{\infty}(M)}<\infty$.
	\end{enumerate}
\end{assumption}

Before proceeding, we give the following definition:
\begin{definition}[$m$-dimensional backward heat kernel]
	Let $\phin\in C^\infty_c([0,1))$ be a cut-off function that we fix hereafter such that $\phin\equiv1$ in $[0,1/2]$, $|\phin'|\le3$ and $0\le\phin\le1$ everywhere. For $R>0$, we let $\hk_R:\RR^{d}\times(-\infty,0)\to\RR$ be defined as
	\begin{equation*}\label{eq:hk_def}
		\hk_R(x,t) = \frac{1}{(4\pi(-t))^{m/2}}\exp\bigg(-\frac{|x|^2}{4(-t)}\bigg)\phin\left(\frac{|x|}{R}\right).
	\end{equation*}
\end{definition}
{A direct calculation yields that there is $C$ universal such that, for every $S\in\Gr(m,d)$ and for every $(x,t)\in\RR^d\times(-\infty,0)$, it holds}
\begin{equation}\label{eq:resto0}
	\left|\de_t\hk_R+\dive_S\nabla\hk_R + \frac{|S^\perp\nabla\hk_R|^2}{\hk_R}   \right|\le C\frac{\chi_{B_R}}{R^{m+2}}.
\end{equation}


\begin{proposition}[Weighted Huisken's monotonicity formula]\label{prop:mono_brakke}
	For every $E_1$, there is $C>0$ with the following property. Let $\Mb$ and $v$ satisfy \Cref{ass:brakke_oscil_interno}, and let $f:B_R\to\RR$ be a non-negative, {convex} function {with $\|\nabla f\|_{\infty}\le 1$}. Then, for every $(x_0,t_0)\in\Sigma_\Mb$, every $r>0$ such that $Q_r(x_0,t_0)\subset Q_R$ and every $-r^2\le t<0$, it holds
	\begin{align}\label{eq:brakke_mono_conclusion}
		\int f\hk_r(\cdot-x_0,t)\dif M_{t_0+t}\ge f(x_0)-C\left(\Lambda+\frac{\eps}{r^2}+\eps\Lambda^2\right)(-t),
	\end{align}
	where $\eps:=\|f\|_{L^\infty(M)}$.
\end{proposition}
\begin{proof}
	Up to rescaling and translating, we may assume $r=1$ and $(x_0,t_0)=(0,0)$; for brevity, we let $\hhk:=\hk_1$.
	By mollification, we may also assume that $f\in C^2(\overline{B_{1}})$. Notice that, in this case, $\dive_S\nabla f \ge0$ for any $S\in\Gr(m,d)$. By \eqref{eq:resto0} and the above inequality, we have
	\begin{align*}
		\de_t(f\hhk) &\le -f\dive_S\nabla\hhk-f\frac{|S^\perp\nabla\hhk|^2}{\hhk}+Cf\chi_{B_1}\\
		&\le \dive_S(\hhk\nabla f-f\nabla\hhk) -f\frac{|S^\perp\nabla\hhk|^2}{\hhk}+Cf\chi_{B_1}.
	\end{align*}
	In particular, for almost every $t$, it holds
	\begin{align}\label{eq:fpsi_derivative}
		\int\de_t(f\hhk)\dif M_t\le\int\left(-H\cdot(\hhk\nabla f - f\nabla\hhk)-f\frac{|(\nabla\hhk)^\perp|^2}{\hhk}\right)\dif M_t + C E_1\eps,
	\end{align}
	where we have used the facts that $M_t\in\Mc_m^2$ for a.e. $t$ and that $M_t(B_1)\le E_1$.
	
	Let now {$I(s) = \int f\hhk(\cdot,s)\dif M_s$}. We use $\phin = f\hhk$ as a test function in \Cref{def:BF}. {By \eqref{eq:fpsi_derivative}}, we obtain, for $-1\le t\le s<0$:
	\begin{align*}
		I(s)-I(t)
		&\le \int_t^s \int \left(-f\frac{|(\nabla\hhk)^\perp|^2}{\hhk} + 2fH\cdot\nabla\hhk-f\hhk|H|^2\right)\dif M_\tau\dif \tau\\
		&\qquad +\int_t^s \int \left( v^\perp\cdot(-f\hhk H + \nabla(f\hhk))\right)\dif M_\tau\dif \tau+CE_1\eps(s-t)\\
		&= \int_t^s \int f\hhk\left(-\left|\frac{(\nabla\hhk)^\perp}{\hhk}-H\right|^2+v^\perp\cdot \left(\frac{(\nabla\hhk)^\perp}{\hhk}-H\right)\right)\dif M_\tau\dif \tau\\
		&\qquad+\int_t^s\int\hhk\nabla f\cdot v^\perp\dif M_\tau\dif \tau + CE_1\eps(s-t),
	\end{align*} 
	where \eqref{eq:ortho} was used in the above equality.
	We then use Young's inequality to bound $v^\perp\cdot\left(\frac{(\nabla\hhk)^\perp}{\hhk}-H\right)\le |v|^2+\left|\frac{(\nabla\hhk)^\perp}{\hhk}-H\right|^2$ and obtain
	\begin{align}
		I(s)-I(t)
		\le\int_t^s\int \left(f\hhk|v|^2 + \hhk|\nabla f||v|\right)\dif M_\tau\dif \tau + CE_1\eps(s-t)\label{eq:difference}.
	\end{align}
	Next, we estimate
	\begin{align*}
		\int f\hhk|v|^2\dif M_\tau\le \Lambda^2\int f\hhk\dif M_\tau\qquad\mbox{and}\qquad \int\hhk|\nabla f||v|\dif M_\tau\le \Lambda\int\hhk\dif M_\tau\le C\Lambda E_1,
	\end{align*}
	where the {latter inequality follows from} \Cref{assumption:MDR} in \Cref{ass:brakke_oscil_interno}.
	Therefore, going back to \eqref{eq:difference}, dividing both sides by $s-t$ and letting $s\searrow t$, we obtain
%
	\begin{equation*}
		I'(t)\le \Lambda^2 I(t) + C(\Lambda+\eps)
	\end{equation*}
	in the sense of distributions, for some $C$ depending also on $E_1$.
	Hence for every $-1\le t\le s<0$ it follows that
	\begin{equation*}
		I(t)\ge I(s) e^{-\Lambda^2(s-t)} - C(\Lambda+\eps)\int_t^se^{\Lambda^2(t-\tau)}\dif\tau\ge I(s) - C\Lambda^2\eps(-t)-C(\Lambda+\eps)(-t)
	\end{equation*}	
	where in the second inequality we used the inequality $e^t\ge1+t$ and the fact that $I(s)\le CE_1\eps$. With the above inequality at hand, it is fairly standard to prove that, for any two sequences $(x_j,t_j)\to0$ and $\tau_j\nearrow0$, it holds
	\begin{equation*}
		\lim_{s\nearrow0} I(s) \ge \limsup_{j\to\infty}\int f(\cdot)\hhk(\cdot-x_j,\tau_j)\dif M_{t_j+\tau_j}.
	\end{equation*}
	By choosing $(x_j,t_j)$ such that $M_{t_j}$ has an approximate tangent plane at $x_j$ and $\tau_j$ converging to $0$ fast enough, we find
	\begin{equation*}
		\limsup_{j\to\infty}\int f(\cdot)\hhk(\cdot-x_j,\tau_j)\dif M_{t_j+\tau_j}\ge f(0),
	\end{equation*}
	as desired.
\end{proof}
\medskip

Similarly to what we did in \Cref{sec:allard}, for $S\in\Gr(m,d)$ and {$X\in \Sigma_\Mb$} we define the quantity
\begin{equation*}
	\osc_S(\Mb, Q_r(X))=\frac{1}{2}\sup\big\{|S^\perp(x-y)|\colon (x,t),(y,s)\in \Sigma_{\Mb}\cap Q_r(X)\big\}.
\end{equation*}

\begin{proposition}[Harnack inequality]\label{prop:harnack_brakke}
	For every $E_1$, there exists {$\eta \in (0,1)$} with the following property. Let $\Mb$ and $v$ satisfy \Cref{ass:brakke_oscil_interno} {with $(0,0) \in \Sigma_\Mb$}. Moreover, assume that, for every $t\in[-R^2,0)$,
	\begin{equation*}
		\int\hk_R(\cdot,t)\dif M_t\le\frac{3}{2},
	\end{equation*}
	that $\osc(Q_R)\le \eta R$ and that $\Lambda R^{2}\le \osc(Q_R)$. Then
	\begin{equation*}
		\osc(Q_{\eta R})\le(1-\eta)\osc(Q_R).
	\end{equation*}
\end{proposition}
\begin{proof}
	Let $\eps=\frac{1}{R}\osc(Q_R)$. We  can thus choose $y_0 \in B_{\eps R}$ such that $\Sigma_\Mb\cap B_R\subset\{y\colon |S^\perp (y-y_0)|\le \eps R\}$. Assume by contradiction that there are two points $Y_i=(y_i,s_i)\in\Sigma\cap Q_{\eta R}$, $i=1,2$, such that $|S^\perp (y_1-y_2)|\ge2(1-\eta)\eps R$.	 Denote $\omega:= \frac{S^\perp (y_1-y_2)}{|S^\perp (y_1-y_2)|}$ and consider 
	\begin{equation*}
		f_1(x) = \left((x-y_0)\cdot \omega -\frac{\eps R}{2}\right)^+\quad \text{and} \quad f_2(x) = \left( - (x-y_0) \cdot \omega -\frac{\eps R}{2}\right)^+\, .
	\end{equation*}
	Consider $\eta \leq \theta \leq 1$ and let $T=-\theta R^2$. Since the $f_i$ are convex and $\|\nabla f_i\|_\infty \leq 1$, by \Cref{prop:mono_brakke} 			
	\begin{align*}\label{eq:right-side}
		\int f_i \hk_{R/2}(\cdot-y_i,T-s_i)\dif M_{T}
		&\ge f_i(y_i)
		-C\left(\Lambda+\frac{\|f_i\|_\infty}{R^2}+\|f_i\|_\infty\Lambda^2\right)(s_i-T)\\
		&\ge f_i(y_i)-3C_0\eps\theta R,
	\end{align*}
	where we have used that $\|f_i\|_\infty\le\eps R$ and $\Lambda\le \eps R^{-1}$.
	By choosing $\theta$ small so that $3C_0\theta\le\frac{1}{10}$ and remarking that $f_i(y_i)\ge\left(\frac{1}{2}-2\eta\right)\eps R$, we obtain
	\begin{equation}\label{eq:right_side}
		\int f_i \hk_{R/2}(\cdot-y_i,T-s_i)\dif M_{T}\ge \left(\frac{2}{5}-2\eta\right)\eps R.
	\end{equation}
	
	Next, we bound from above the left-hand side of \eqref{eq:right_side}: provided $\eta^2$ is much smaller than $\theta$ chosen above, it holds $T-s_i\le-\frac{\theta}{2}R^2$, hence there is some constant $L$ such that
	\begin{equation*}
		\left|\hk_{R/2}(x-y_i,T-s_i)-\hk_{R/2}(x,T)\right|\le L\left(R^{-m-1}|y_i|+R^{-m-2}|s_i|\right)\le L \eta R^{-m}
	\end{equation*}
	for every $x\in\RR^d$. Therefore
	\begin{equation}\label{eq:left_side}
		\begin{split}
		\int f_i \hk_{R/2}(\cdot-y_i,T-s_i)\dif M_{T}
		&\le\frac{\eps R}{2}\int_{\supp{f_i}}\hk_{R/2}(\cdot-y_i,T-s_i)\dif M_{T}\\
		&\le\frac{\eps R}{2}\bigg(\int_{\supp{f_i}}\hk_{R/2}(\cdot,T)\dif M_{T} + LE_1\eta\bigg)
		\end{split}
	\end{equation}
	where we have used the fact that $\frac{M_T(B_R)}{R^m}\le E_1$.
	
	If $\eta$ is smaller than some universal constant, then $\supp f_1\cap\supp f_2=\emptyset$, thus summing \eqref{eq:right_side} and \eqref{eq:left_side} for $i=1,2$, we obtain
	\begin{align*}
		\int\hk_{R/2}(\cdot,T)\dif M_{T}
		&\ge\left(\int_{\supp{f_1}}\hk_{R/2}(\cdot,T)\dif M_{T}+\int_{\supp{f_2}}\hk_{R/2}(\cdot,T)\dif M_{T}\right)\\
		&\ge 2\,\frac{2}{\eps R}\left(\frac{2}{5}-2\eta\right)\eps R -2\,LE_1\eta\\
		&=\frac{8}{5}-C\eta
	\end{align*}
	for some $C$ depending on $E_1$. Choosing $\eta$ smaller, if needed, contradicts the assumption that $\int\hk_{R/2}(\cdot,T)\dif M_T\le\frac{3}{2}$, thus concluding the proof.
\end{proof}
\medskip
As \Cref{harnack_allard} implies \Cref{prop:allard_decay_final}, we obtain the following result as a corollary of \Cref{prop:harnack_brakke}.
\begin{proposition}[Decay of oscillations]\label{cor:osc_decay_brakke}
	For every $E_1$, there exist $C$ and $\beta$ with the following property. Let $\Mb$ be a Brakke flow with transport term $v$ in $Q_R$ that satisfies \Cref{ass:brakke_oscil_interno} {with $(0,0) \in \Sigma_\Mb$} and such that, for every $t\in[-R^2,0)$,
	\begin{equation*}
		\int\hk_R(\cdot,t)\dif M_t\le\frac{3}{2}.
	\end{equation*}
	Then, for every $r\ge C(\osc(Q_R)+\Lambda R^{2})$, it holds
	\begin{equation*}
		\osc(Q_{r})\le C(\osc(Q_R)+\Lambda R^{2}) \left(\frac{r}{R}\right)^\beta.
	\end{equation*}
\end{proposition}
%

\subsection{Brakke's theorem}
The present subsection is dedicated to the proof of the following version of Brakke's regularity theorem:
\begin{theorem}[Interior regularity]\label{thm:brakke_interior}
	For every $\alpha\in(0,1)$ and every $E_1$, there are positive and small constants $\theta$ and $\delta_0$ with the following property.
	Let $\Mb$ be a Brakke flow with transport term $v$ in $Q_1$. Assume that:
	\begin{itemize}
		\item $(0,0)\in\Sigma_\Mb$;
		\item for every $t\in[-1,0]$ and every $B_r(x)\subset B_1$, $M_t(B_r)\le E_1 r^m$;
		\item $\int_{B_1}\hk(\cdot,-\theta^2)\dif M_{-\theta^2}\le1+\delta_0$; 
		\item $\|v\|_{L^\infty(M)}\le\delta_0$.
	\end{itemize}
	Then $\Sigma_\Mb\cap Q_{\theta/8}$ is the graph of some function $u\in C^{1,\alpha}(Q_{\theta/8}^m;\RR^{d-m})$ with
	$$||u||_{C^{1,\alpha}}\le C\bigg(\inf_{S\in\Gr(m,d)}\osc_S(M,Q_1)+\|v\|_{L^\infty(M)}\bigg).$$
\end{theorem}

	In the above statement, by $u\in C^{1,\alpha}(Q_{\theta/8}^m;\RR^{d-m})$ we mean that $(x,t)\mapsto u(x,t)$ is differentiable with respect to $x$ and that there is $C>0$ such that, for every $(x',t),(y',s)\in Q_{\theta/8}^m$ it holds
	\begin{equation*}
		|u(x,t)-u(y,s)-\nabla u(x,t)\cdot (y-x)|\le C\left(|x-y|^2+|t-s|\right)^{\frac{1+\alpha}{2}}.
	\end{equation*}
	$\|u\|_{C^{1,\alpha}}$ corresponds(up to a multiplicative constant) to the smallest $C$ for which the above inequality holds.
\begin{remark}
	This result is analogous to the \enquote{end-time regularity} proved in \cite{stuvard-tonegawa}, although our proof requires the forcing term $v$ to be a $L^\infty$ function. We also remark that, differently from the case of minimal varifolds, higher regularity on $v$ does not straightforwardly yield higher regularity for $\Sigma_\Mb$. In this regard, see \cite{tonegawa_second_der}.
\end{remark}

\medskip

Similarly to \Cref{thm:allard_main}, \Cref{thm:brakke_interior} follows from an \textit{improvement of flatness}, which we state and prove next.
We first make some preliminary assumptions:
\begin{assumption}\label{ass:brakke_IOF_section}
	\hfill
	
	For some $X_0=(x_0,t_0)$ and $R>0$:
	\begin{enumerate}
		\item $\Mb$ is a Brakke flow with transport term $v$ in $Q_R(X_0)$.
		\item \label{ass:0inSigma_IOF_brakke} $X_0\in\Sigma_\Mb$.
		\item \label{ass:double_brakke_IOF}For every $(x,t)\in Q_{R/2}(X_0)$ and every $-\frac{R^2}{4}\le\tau<0$, it holds
		\begin{equation*}
			\int_{B_{R/2}(x)}\hk(\cdot-x,\tau)\,dM_{t}\le\frac{3}{2}.
		\end{equation*}
	\end{enumerate}
\end{assumption}

\begin{theorem}[Improvement of flatness]\label{thm:brakke_IOF_NOB}
	For every $\alpha\in(0,1)$, there are universal constants $\eps_0, \eta$ (small) and $C$ (large) with the following property.
	Let $\Mb$ satisfy \Cref{ass:brakke_IOF_section}. Assume, in addition, that for some $\eps\le\eps_0$ and $S\in\Gr(m,d)$:
	\begin{equation}\label{eq:both_conditions}
		\osc_S(\Mb,Q_R(X_0))\le \eps R\qquad\mbox{ and }\qquad C\|v\|_{L^\infty(M)}\le\frac{\eps}{R}.
	\end{equation}
	Then there exists $T\in\Gr(m,d)$ with $|T-S|\le C\eps$ such that
	\begin{equation}\label{eq:tesi_brakke_iof_nobdry}
	\osc_T(\Mb,Q_{\eta R}(X_0))\le \eta^{1+\alpha}\eps R.
	\end{equation}
\end{theorem}
\begin{proof}
	By rescaling and translating, we may assume $R=1$ and $X_0=(0,0)$.
	We argue by contradiction and compactness: assume there exist a sequence $\eps_j\searrow0$ and two sequences $\{\Mb^j\}$ and $\{v^j\}$ such that $\Mb^j$ is a Brakke flow with transport $v^j$ for which the assumptions of the theorem are satisfied with $\eps$ replaced by $\eps_j$.
	For brevity, we let $\Sigma^j:=\Sigma_{\Mb^j}$.
	
	\textbf{Step 1: compactness and convergence to a plane.} By \Cref{prop:brakke_compactness} below, up to extracting a subsequence (which we do not relabel), $\{\Mb^j\}$ converges to a Brakke flow without transport term $\Mb^\infty$. Moreover, $M^\infty_t(S^\compl)=0$ for every $t$, hence $M^\infty_t = \Theta^m(M^\infty_t,\cdot)\Hc^m\rest S$. Now, for almost every $t$, $M^\infty_t\in\Mc_m^2$ and, by \Cref{ass:double_brakke_IOF} in \Cref{ass:brakke_IOF_section} and by Huisken's monotonicity, $\Theta^m(M^\infty_t,x)\le E_1$ for every $x$. In particular, by \eqref{eq:bounded_first_var}, for every $\phi\in C^\infty_c(S;S)$ it holds
	\begin{equation*}
		\int_S\Theta^m\dive\phi = \int\dive_S\tilde\phi\dif M^\infty_t \le C\left(\int|\phi|^2\dif M^\infty_t\right)^{1/2}\le C \left(E_1\int_S|\phi|^2\right)^{1/2},
	\end{equation*}
	where $\tilde\phi\in C^\infty_c(\RR^d;\RR^d)$ is any extension of $\phi$ such that $||\tilde\phi||_\infty\le||\phi||_\infty$. Hence $S\ni x\mapsto\Theta^m(M^\infty_t,x)$ is a locally $W^{1,2}$ function and, by \Cref{def:BF}, it is integer-valued. Hence, for almost every $t$, either $\Theta^m(M^\infty_t,\cdot)\equiv0$ or $\Theta^m(M^\infty_t,\cdot)\ge1$ for $\Hc^m$-almost every $x$. However, $(0,0)\in\Sigma_{\Mb^\infty}$ hence, by \eqref{eq:def_brakke}, it cannot be that $M^\infty_t=0$ for some $t<0$. Therefore $M^\infty_t\ge\Hc^m\rest S$ for almost every $t\in[-1,0]$.
	
	\textbf{Step 2: Hausdorff convergence. }Let now
	\begin{gather*}
		\tSigma^j=\left\{\left(Sx,\frac{1}{\eps_j}S^\perp x,t\right)\colon (x,t)\in\Sigma^j\right\}\subset B_1^S\times B_1^{S^\perp}\times[-1,0].
	\end{gather*}
	By \Cref{ass:0inSigma_IOF_brakke} in \Cref{ass:brakke_IOF_section}, $\tSigma^j\neq\emptyset$. Therefore, up to subsequences, $\{\tSigma^j\}$ converges in the Hausdorff distance to some relatively closed set $\tSigma\subset B_1^S\times B_1^{S^\perp}\times[-1,0]$.
	
	\newcommand{\ub}{\mathbf{u}}
	
	\textbf{Step 3: $\tSigma$ is a graph. }Let
	\begin{equation*}
		\ub(x',t)=\{y\in B_1^{S^\perp}\colon (x',y,t)\in\tSigma\}.
	\end{equation*}
	Arguing as in the proof of \Cref{thm:IOF_allard}, by Step 1 we first show that $\ub(x',t)\neq\emptyset$ for every $(x',t)\in Q_1^S$. Next, by \Cref{cor:osc_decay_brakke}, we conclude that $\ub(x',t)=\{u(x',t)\}$ is a singleton for every $(x',t)\in Q_{1/2}^S$ and that $u$ is \holder\ continuous.
	
	\textbf{Step 4: conclusion. }Arguing as in \Cref{lemma:uisharmonic}, \Cref{prop:max_brakke} below yields that $u|_{Q^S_{1/4}}$ is a solution to the heat equation. The desired result follows by the Hausdorff convergence established in Step 2 and by classical Schauder estimates for the heat equation.
\end{proof}

The following result was used in the proof of \Cref{thm:brakke_IOF_NOB}.
\begin{proposition}[Compactness]\label{prop:brakke_compactness}
	Let $\{\Mb^j\}, \{v_j\}$ be two sequences such that, for each $j$, $\Mb^j$ is a Brakke flow with transport term $v_j$ in $Q_R$. Assume that, for every $W\compact B_R$,
	\begin{equation*}
		\sup_{t\in[-R^2,0]}\sup_{j\in\NN}M_t(W)<\infty
	\end{equation*}
	and that
	\begin{equation*}
		\|v_j\|_{L^\infty(M^j)}\to0.
	\end{equation*}
	Then there exist a subsequence $\{j_\ell\}\subset\NN$ and a Brakke flow $\Mb$ without transport term in $Q_R$ such that, for every $t\in[-R^2,0]$, $M^{j_\ell}_t\weakly M_t$
	as Radon measures in $B_R$.
\end{proposition}
\begin{proof}
	The proof, in the case of Brakke flows without transport term, can be found in \cite[\S7]{ilmanen_thesis}: the more general case can be proved by straightforward modifications. See also \cite[Section 4.2.2]{gaspa_thesis}.
\end{proof}

We conclude the present section by sketching the proof of \Cref{thm:brakke_interior}.
\begin{proof}[Proof of \Cref{thm:brakke_interior}]
	The proof is analogous to that of \Cref{thm:allard_main}, thus we will only highlight the relevant differences.
	
	\textbf{Step 1 . }We show that the assumptions of \Cref{thm:brakke_IOF_NOB} are in place for every $X_0\in\Sigma_{\Mb}\cap Q_{3\theta/4}$ and $R=\theta/4$, provided $\theta$ and $\delta_0$ in the statement of \Cref{thm:brakke_interior} are chosen small enough.
	In order to do so, we argue by contradiction. For $E_1$ and $\alpha$ fixed, assume there is a sequence $\{\Mb^j\}$ of Brakke flows with transport term $v_j$ that satisfy the assumptions of \Cref{thm:brakke_interior} with $\delta_0$ and $\theta$ replaced by some $\delta_j\searrow0$ and $\theta_j\searrow0$, respectively. By \Cref{prop:brakke_compactness} above, the rescalings
	\begin{equation*}
		\tilde\Mb^j = \left\{\tilde M^j_t = \theta_j^{-m}(\mu_{\theta_j})_\sharp M^j_{\theta_j^2t}\right\}_{t\in[-\theta_j^{-2},0]}
	\end{equation*}
	(where $\mu_r(y) = \frac{y}{r}$) converge to a Brakke flow $\tilde\Mb$ without transport term in $\RR^d\times(-\infty,0]$. By Huisken's monotonicity formula, $\tilde\Mb$ must be a stationary unit-density plane. Therefore the assumptions of \Cref{thm:brakke_IOF_NOB} are satisfied for $j$ large enough.
	
	\textbf{Step 2. }With minor modifications from Step 2 of the proof of \Cref{thm:allard_main}, one shows that $\Sigma_\Mb\cap Q_{3\theta/4}$ is the graph of some function $u:S(\Sigma_\Mb\cap Q_{3\theta/4})\to\RR^{d-m}$ and that, for every $(x',t)\in S(\Sigma_\Mb\cap Q_{3\theta/4})$, there is $L_{(x',t)}$ so that, for every $(y',s)\in S(\Sigma_\Mb\cap Q_{3\theta/4})$, it holds
	\begin{equation}\label{eq:infatti_holder}
		|u(y',s)-u(x',t)-L_{(x',t)}(y'-x')|\le C\eps(|x'-y'|^2+|t-s|)^{\frac{1+\alpha}{2}},
	\end{equation}
	where $\eps=\osc_S(\Sigma_\Mb,Q_1)+||v||_{L^\infty}$.
	
	\textbf{Step 3. }In order to prove that $S(\Sigma_\Mb\cap Q_{3\theta/4})$ covers all of $Q_{\theta/2}^S$, we argue by contradiction. Assume there is $(x_0',t_0)\in Q_{\theta/2}^S\setminus S(\Sigma_\Mb)$ with $t_0<0$. Then there exists a smooth curve $p:[t_0,0]\to B_{\theta/2}$ and $\rho>0$ with the following properties: 
	\begin{equation*}
		\begin{cases}
			p(t_0)=x_0'\mbox{ and }p(0)=0;\\
			Q_\rho^S(p(t),t)\subset Q_{\theta/2}^S;\\
			Q_\rho^S(x_0',t_0)\subset S(\Sigma_\Mb)^\compl.
		\end{cases}
	\end{equation*}
	For brevity, we let $Q(t):=Q^S_\rho(p(t),t)$.
	Let now
	\begin{equation*}
		\bar t  = \inf\{t\colon Q(t)\cap S(\Sigma_\Mb)\neq\emptyset\}.
	\end{equation*}	
	Then $\mathrm{Int}Q(\bar t)\cap S(\Sigma_\Mb)=\emptyset$ and there is $(y_0,s_0)\in\Sigma_\Mb$ such that $(y_0',s_0)\in\partial Q(\bar t)$.
	By continuity of $p$, we exclude the case $y_0'\in B_\rho^S(p(\bar t))$ and $s_0=\bar t-\rho^2$. Moreover, if it were $y_0'\in B_\rho^S(p(\bar t))$ and $s_0=\bar t$, then by monotonicity (see, for instance, \cite[Proposition 3.6]{Tonegawabook}) it would be $Q(t)\cap S(\Sigma_\Mb)\neq\emptyset$ for $t<\bar t$ close enough to $\bar t$, which would contradict the choice of $\bar t$.
	Therefore $\bar t-\rho^2\le s_0\le\bar t$ and $y_0'\in\partial B_\rho^S(p(\bar t))$.
	In particular,
	\begin{equation}\label{eq:outside_half_space}
		S(\Sigma_\Mb)\cap\{s<s_0\}\subset \bigcup_{s<s_0}\left(B_\rho^S(p(s+\bar t-s_0))\right)^\compl\times\{s\}.
	\end{equation}
	
	Let us now consider a sequence $r_j\searrow0$ and define the dilations
	\begin{equation}\notag
		\Mb^j = \left\{M^j_t = r_j^{-m}(\mu_{r_j,y_0})_\sharp M_{s_0+r_j^2t}\right\}_{t\in[-r_j^{-2},0]}
	\end{equation}
	(where $\mu_{r,x_0}(y)=\frac{y-x_0}{r}$).
	\Cref{prop:brakke_compactness} yields that, up to subsequences, $\Mb^j$ converges to a limit Brakke flow $\Mb^\infty$ (without transport term).
	Then \eqref{eq:infatti_holder} and \eqref{eq:outside_half_space} imply that there exists a $m$-dimensional half plane $T^+$ such that
	\begin{equation*}
		\Sigma_{\Mb^\infty}\subset T^+\times(-\infty,0].
	\end{equation*}
	Moreover, since $(y_0,s_0)\in\Sigma_\Mb$, we have $(0,0)\in\Sigma_{\Mb^\infty}$.
	We conclude by showing that this violates \Cref{prop:max_brakke}. Up to a change of coordinates, say $T^+=\{x_{m+1}=\dots=x_d=0\mbox{ and }x_m>0\}$ and let
	\begin{equation}\notag
		f(x,t)= \frac{|T^\perp x|^2}{2}-\frac{|x''|^2}{2m}+\frac{|x_m|^2}{2}-x_m+\frac{1}{2m}t,
	\end{equation}
	where $x'' = (x_1,\dots,x_{m-1})$.
	Then $f|_{\Sigma_{\Mb^\infty}\cap \{t\le0\}}$ has a local maximum at $(0,0)$. However, it holds $\de_tf(0,0) = \frac{1}{2m}$ and
	\begin{equation*}
		\trace_m D^2f(0,0) = 1-(m-1)\frac{1}{m} = \frac{1}{m},
	\end{equation*}
	which contradicts \Cref{prop:max_brakke}, thus proving the desired result.
\end{proof}

\appendix
\section{Maximum principles}
For the sake of completeness, we prove two maximum principles which were used in the proofs of Theorems \ref{thm:IOF_allard} and \ref{thm:brakke_IOF_NOB}. {The proofs of \Cref{prop:max_principle} and \Cref{prop:max_brakke} are inspired by \cite{white_areablowup} and \cite[Theorem 14.1]{white_notes}, respectively.} In what follows, for a symmetric matrix $A$, we let $\trace_m A$ denote the sum of the $m$ smallest eigenvalues of $A$.
\begin{proposition}\label{prop:max_principle}
	Let $M\in\Mc_m^\infty(U)$ be such that $\Lambda:=\|H\|_{L^\infty(M)}<\infty$. If $f\in C^2(U)$ is such that $f|_{\supp M}$ achieves a local maximum at $x_0\in\supp M\cap \mathrm{Int}\,U$, then 
	\begin{equation*}
		\trace_m D^2f(x_0)\le \Lambda|\nabla f(x_0)|.
	\end{equation*}
\end{proposition}
\begin{proof}
	Without loss of generality, we assume that $x_0=0$ and that $f$ has a global strict maximum at $0$.	By contradiction, assume that there is $r>0$ such that
	\begin{equation*}
		\trace_mD^2f(x)- \Lambda|\nabla f(x)|\ge\delta
	\end{equation*}
	for every $x\in B_r$.
	Up to choosing $r$ smaller and up to adding a constant to $f$, we also assume that $f(0)>0$ and that $\{f>0\}\cap\supp M\subset B_r$.
	By mollification, \eqref{eq:first_var} and the assumption $\|H_M\|_\infty\le\Lambda$ give
	\begin{align*}
		\Lambda\int_{\{f>0\}}f|\nabla f|\dif M&\ge\int_{\{f>0\}}\dive_{T_xM}(f\nabla f)\dif M \\
		&= \int_{\{f>0\}}\left(\left|\nabla_{T_xM} f(x)\right|^2 + f(x)\dive_{T_xM}\nabla f(x)\right)\dif M(x)\\
		&\ge \int_{\{f>0\}} f\dive_{T_xM}\nabla f\dif M
	\end{align*}
	where $\nabla_{T_xM} f(x)$ is the projection of $\nabla f(x)$ onto $T_xM$.
	This yields
	\begin{equation*}
		\delta\int_{\{f>0\}}f\dif M\le\int_{\{f>0\}}f\big(\dive_{T_xM}\nabla f-\Lambda|\nabla f|\big)\dif M\le0,
	\end{equation*}
	which contradicts the fact that $0\in\supp M$.
\end{proof}

\begin{proposition}\label{prop:max_brakke}
	Let $\Mb$ be a Brakke flow with transport term $v$ in $U\times[-\Omega,0]$ and assume that $\Lambda:=\|v\|_{L^\infty(M)}<\infty$. If $f\in C^2(U\times[-\Omega,0])$ is such that $f|_{\Sigma_\Mb}$ has a local maximum at $X_0=(x_0,t_0)\in\mathrm{Int}\,U\times(-\Omega,0]$, then
	\begin{equation*}
		\trace_m D^2f(X_0)-\de_tf(X_0)\le \Lambda|\nabla f(X_0)|.
	\end{equation*}
\end{proposition}
\begin{proof}
		\newcommand{\backball}{\pball^-}
		Up to a translation, we assume that $(x_0,t_0)=(0,0)$. Moreover, we assume that $f|_{\Sigma\cap\{t\le0\}}$ has a strict local maximum at $(0,0)$.
		Assume the result does not hold. In particular, since $f\in C^2$, we can choose $\rho>0$ and $\eps>0$ small so that:
		\begin{enumerate}
			\item \label{point:first} $\de_tf-\trace_m D^2f+\Lambda|\nabla f|<-\eps<0$ in $Q_\rho$;
			\item \label{point:last} $0<f(0,0)\le\frac{2\eps}{\Lambda^2}$ and $f<0$ in $\Sigma\cap\{t\le0\}\setminus Q_\rho$;
		\end{enumerate}
		where the second point holds up to adding a constant to $f$. 
		We now let $\phi(x,t) = (f^+(x,t))^4$, where $f^+=\max\{f,0\}$, and we use $\phi$ as a test function for \eqref{eq:def_brakke}. Since $\phi(\cdot, -\rho^2)=0$ by assumption, we have
		\begin{equation}\label{eq:max_princ_chain}
			\begin{split}
			0&\le\int\phi(\cdot,0)\dif M_{0} - \int\phi(\cdot,-\rho^2)\dif M_{-\rho^2}\\
			&\le\int_{-\rho^2}^0\int\left(\de_t\phi + (-\phi H+\nabla\phi)\cdot(H+v^\perp)  \right)\dif M_t\dif t.
			\end{split}
		\end{equation}
		We now have
		\begin{equation}\notag
			\int H\cdot\nabla\phi\dif M_t = -\int\dive_{T_xM_t}\nabla\phi\dif M_t
		\end{equation}
		for a.e. $t$, and we may bound $-\phi H\cdot v^\perp\le\phi|H|^2+\phi|v|^2$. Therefore \eqref{eq:max_princ_chain} gives
		\begin{equation}\label{eq:A2}
			0\le\int_{-\rho^2}^0\int\left(\de_t\phi-\dive_{T_xM_t}\nabla\phi+\phi\Lambda^2+|\nabla\phi|\Lambda  \right)\dif M_t\dif t.
		\end{equation}
		Straightforward computations give
		\begin{gather*}
			|\nabla\phi|= 4(f^+)^3|\nabla f|,\qquad\qquad
			\dive_{T_xM_t}\nabla\phi = 4(f^+)^3\dive_{T_xM_t}\nabla f\ge4(f^+)^3\trace_mD^2f
		\end{gather*}
		and $\de_t\phi= 4(f^+)^3\de_tf$. Therefore \eqref{eq:A2} reads
		\begin{align}\label{eq:max_almost}
			0\le\int_{-\rho^2}^0\int4(f^+)^3\left(\de_tf-\trace_mD^2f+|\nabla f|\Lambda+\Lambda^2\frac{f^+}{4}\right)\dif M_t\dif t .
		\end{align}
		Since $f$ has a maximum at $0$, it holds $\Lambda^2\frac{f^+}{4}\le \Lambda^2\frac{f(0,0)}{4}\le\frac{\eps}{2}$, hence
		\begin{equation*}
			\de_tf-\trace_mD^2f+|\nabla f|\Lambda+\Lambda^2\frac{f^+}{4}\le -\eps + \frac{\eps}{2} = -\frac{\eps}{2}
		\end{equation*}
		in $Q_\rho\cap \{f\ge0\}$.
		The latter inequality and \eqref{eq:max_almost} yield $M(Q_\rho)=0$, which contradicts the assumption $(0,0)\in\Sigma_\Mb$.
\end{proof}


\end{document}